\documentclass{amsart}

\usepackage{amsmath, amssymb, amsthm}
\usepackage{graphicx,color}
\usepackage[dvipsnames]{xcolor} %
\usepackage{mathtools}%
\usepackage{bm}%

\usepackage{xargs}%
\usepackage[utf8]{inputenc} %
\usepackage[T1]{fontenc} %
\usepackage{lmodern} %

\usepackage[a4paper, left=2.5cm, right=2.5cm]{geometry} %

\usepackage{caption}
\usepackage{subcaption}

\usepackage[colorlinks,breaklinks]{hyperref} %
\hypersetup{
citecolor=green!50!black,
linkcolor=red!50!black,
}%

\usepackage[noabbrev,capitalise]{cleveref} %
\usepackage[]{todonotes}

\usepackage{nicefrac}%

\usepackage[shortlabels,inline]{enumitem} %

\usepackage[final]{showlabels} %
\usepackage{nameref}
\makeatletter
\let\orgdescriptionlabel\descriptionlabel
\renewcommand*{\descriptionlabel}[1]{%
  \let\orglabel\label
  \let\label\@gobble
  \phantomsection
  \edef\@currentlabel{#1\unskip}%
  \let\label\orglabel
  \orgdescriptionlabel{#1}%
}
\makeatother

\makeatletter
\@namedef{subjclassname@2020}{\textup{2020} Mathematics Subject Classification}
\makeatother

\usepackage{tikz}
\usepackage{tikz-cd} %
\usetikzlibrary{decorations.pathreplacing,decorations.markings,arrows,patterns}

\newtheorem{theorem}{Theorem}[section] 
 
\newtheorem{lemma}[theorem]{Lemma} 
\newtheorem{corollary}[theorem]{Corollary}

\newtheorem*{theorem*}{Theorem}%

\theoremstyle{remark}
\newtheorem{remark}[theorem]{Remark}

\newtheorem{example}[theorem]{Example}

\theoremstyle{definition}
\newtheorem{definition}[theorem]{Definition}

\newcommand{\defn}[1]{\emph{\color{blue} #1}} %

\newcommand{\floor}[1]{\left\lfloor {#1} \right\rfloor}

\newcommand{\ffloor}[2]{\left\lfloor{\frac{#1}{#2}}\right\rfloor} %

\newcommand{\set}[2]{\ensuremath{\left\{#1\,\middle|\,#2\right\}}}

\newcommand{\wh}{\widehat}
\makeatletter
\newcommand{\wwh}[1]{%
\begingroup%
  \let\macc@kerna\z@%
  \let\macc@kernb\z@%
  \let\macc@nucleus\@empty%
  \widehat{\raisebox{.35ex}{\vphantom{\ensuremath{#1}}}\smash{\widehat{#1}}\,}
  \hskip -1.5pt%
\endgroup%
}
\makeatother

\def\RR{\mathbb{R}}

\def\cT{\mathcal{T}}

\def\cC{\mathcal{C}}
\def\cP{\mathcal{P}}

\newcommand{\p}[1][p]{#1} %
\newcommand{\pc}[1][P]{#1} %
\newcommand{\tr}[1][T]{#1} %
\newcommand{\regtriang}[1]{\cT(#1)} %

\newcommand{\conv}{\operatorname{conv}}

\newcommand\restr[2]{{%
  \left.\kern-\nulldelimiterspace %
  #1 %
  \vphantom{\big|} %
  \right|_{#2} %
  }}%
\newcommand{\ball}[2]{B(#1, #2)}

\newcommand{\size}[1]{c(#1)}

\title{Many regular triangulations and many polytopes}
\date{\today}

\thanks{Work of Padrol and Philippe is supported by grants ANR-17-CE40-0018 and
ANR-21-CE48-0020 of the French National Research Agency ANR (projects CAPPS and PAGCAP). Work of Padrol and Santos is supported by grant PID2019-106188GB-I00 of MCIN/AEI/10.13039/501100011033 and project CLaPPo (21.SI03.64658) of Universidad de Cantabria and Banco Santander.}

\author{Arnau Padrol}
\address[Arnau Padrol]
{Departament de Matem\`atiques i Inform\`atica, Universitat de Barcelona, Barcelona, Spain.}
\email{arnau.padrol@ub.edu}
\urladdr{\url{https://www.ub.edu/comb/arnaupadrol}}

\author{Eva Philippe}
\address[Eva Philippe]{Sorbonne Université and Université de Paris, CNRS, IMJ-PRG, F-75005 Paris, France.}
\email{eva.philippe@imj-prg.fr}
\urladdr{\url{https://perso.imj-prg.fr/eva-philippe}}

\author{Francisco Santos}
\address[Francisco Santos]{Dept. de Matem\'aticas, Estad\'istica y Computaci\'on, Universidad de Cantabria, Santander, Spain}
\email{francisco.santos@unican.es}
\urladdr{\url{https://personales.unican.es/santosf}}

\begin{document}

\maketitle

\begin{abstract}
We show that for fixed $d>3$ and $n$ growing to infinity there are at least $(n!)^{d-2 \pm o(1)}$ different labeled combinatorial types of $d$-polytopes with $n$ vertices. This is about the square of the previous best lower bounds. As an intermediate step, we show that certain neighborly polytopes (such as particular realizations of cyclic polytopes) have at least $(n!)^{ \floor{(d-1)/2}  \pm o(1)}$
regular triangulations.
\end{abstract}

\section{Introduction}

A \defn{polytope} is the convex hull of a finite set of points in a real Euclidean space. Its \defn{combinatorial type} is given by its poset of \defn{faces} (subsets of the polytope maximized by linear functionals, ordered by inclusion). 
In the preface of his now classical book in polytope theory~\cite{G03}, Gr{\"u}nbaum 
traces the problem of enumerating the number of combinatorial types of polytopes back to Euler, and cites its difficulty as one of the main reasons for the ``\emph{decline in the interest in convex polytopes}'' at the beginning of the XXth century. 

These efforts were concentrated in the case of $3$-dimensional polytopes, starting with many contributions by Cayley and Kirkman, according to Gr{\"u}nbaum's historical remarks in~\cite[Chapter~13.6]{G03}. Thanks to Steinitz's Theorem, which gives a correspondence between combinatorial types of $3$-dimensional polytopes and $3$-connected planar graphs, nowadays we have quite precise knowledge on the number of $3$-polytopes with $n$ vertices~\cite{BenderWormald1988,RichmondWormald1982} and the distribution of many combinatorial parameters~\cite{BenderGaoRichmond1992}. 

In contrast, for higher-dimensional polytopes the problem is still very far from being solved. One of the main difficulties lies in the lack of a combinatorial characterization of face lattices of polytopes. Mn\"ev's Universality Theorem~\cite{Mnev1988} and its extension by Richter-Gebert~\cite{RichterGebert1996}, imply that deciding whether a poset is the face lattice of a $4$-dimensional polytope is computationally hard ($\exists\RR$-complete). It seems thus that a simple combinatorial characterization is impossible, which is one of the intrinsic difficulties of the enumeration problem. The problem remains hard even when restricting to the ``generic'' case of \defn{simplicial} polytopes, where all faces except the whole polytope are simplices (equivalently, polytopes whose combinatorial type does not change when the vertices are perturbed), see~\cite{AdiprasitoPadrol2017}.

However, the mere number of polytopes is  relatively small.
In 1986 Goodman and Pollack~\cite{GP86} showed that the number of (labeled) combinatorially different simplicial $d$-polytopes with $n$ vertices is bounded by $(n!)^{c_d}$ for some constant $c_d$ depending solely on $d$, and Alon~\cite{Alon1986} proved that this upper bound is valid for non-necessarily simplicial polytopes too. This contrasts with the number of combinatorially different simplicial $(d-1)$-spheres with $n$ vertices, which grows at least as $e^{\Omega(n^{\floor{d/2}})}$~\cite{Kalai88,NSS16}.

In 1982 Shemer~\cite{Shemer82} had devised constructions producing about $ (n!)^{\frac12 \pm o(1)}$ different simplicial polytopes. This matches the upper bound, except for the fact that the constant $c_d$ in the upper bound of Goodman and Pollack and Alon is ${ d^2\pm o(1)}$, much bigger than the $1/2$ obtained by Shemer.
The construction was greatly improved by Padrol~\cite{Pad13} (see also~\cite{GP16}) who showed that there are at least $(n!)^{\floor{d/2}\pm o(1)}$ (labeled) neighborly polytopes. There are alternative constructions that give these many different combinatorial types of polytopes, which led Nevo and Padrol to ask whether the number of $d$-dimensional polytopes with $n$ vertices and $m$ facets was bounded above by $m^{n+o(n)}$ (unpublished). As the maximal number of facets of a $d$-polytope with $n$ vertices is $O(n^{\floor{d/2}})$ by the Upper Bound Theorem~\cite{McM70}, this would imply that the bound of $(n!)^{\floor{d/2}\pm o(1)}$ is asymptotically tight. 

The main result in this paper gives a negative answer to this question, by essentially doubling the exponent of  $n!$ in the construction of Padrol:

\begin{theorem}\label{thm:manypolytopes}
The number of different labeled combinatorial types of $d$-polytopes with $n$ vertices for fixed $d>3$ and $n$ growing to infinity is at least $(n!)^{d-2 \pm o(1)}$.
\end{theorem}

All the polytopes that we construct are  $\lfloor(d-1)/2\rfloor$-neighborly. That is, they are neighborly for odd~$d$, but only $\left(\frac{d}{2}-1\right)$-neighborly if $d$ is even. In fact, for even $d$ the number of neighborly polytopes in our family is at most the same as in the family constructed by Padrol.
See Remark~\ref{rmk:notmanyneighborly} for more details.

\medskip

Enumerating polytopes is intimately tied to enumerating \defn{regular} triangulations of point configurations; that is, triangulations arising as lower envelopes of polytopes of one  dimension more.
In fact, the number of (combinatorial types of) simplicial $d$-polytopes with $n$ vertices coincides with that of $(d-1)$-dimensional regular triangulations with $n-1$ vertices. See the beginning of Section~\ref{sec:manypolytopes} for details on this relation.
In the same vein, counting all triangulations, regular or not, is related to counting simplicial spheres. 

In particular, the Goodman-Pollack bound implies the same upper bound of $(n!)^{d^2\pm o(1)}$ for the number of  regular triangulations, while the construction of Kalai~\cite{Kalai88} can be adapted to derive that the cyclic $d$-polytope with $n$ vertices has at least~$e^{\Omega(n^{\floor {d/2}})}$ triangulations in total~\cite[Theorem 6.1.2]{DRS10}. 

Observe that  the upper bound is for the total number of (combinatorially different) regular triangulations of \emph{all} polytopes (for fixed parameters $n$ and $d$), while the construction of Kalai counts triangulations of \emph{a single polytope}. 
For regular triangulations of a single polytope, it is shown in \cite[Theorem 7.2.10]{DRS10} that the Cartesian product of a cyclic $3$-polytope with $n$ vertices and a segment has at least $(n/2)! = (n!)^{1/2 \pm o(1)}$ regular triangulations.  The second result in this paper is a significant improvement of this lower bound, showing  for example that:

\begin{theorem}\label{thm:manytriangulations}
For fixed $d\ge3$ and $n$ going to infinity, there are realizations of the cyclic $d$-polytope with $n$ vertices having at least 
\[
(n!)^{  \ffloor{d-1}{2}  \pm o(1)}
\]
regular triangulations.
\end{theorem}

It has to be noted that the total number of triangulations of a polytope (or point configuration) depends only on its \defn{oriented matroid} (another combinatorial invariant that is finer than the combinatorial type), while the number of regular triangulations varies for different realizations of the same oriented matroid.

Apart of its intrinsic interest, 
Theorem~\ref{thm:manytriangulations} is an intermediate step for Theorem~\ref{thm:manypolytopes}; the proof of Theorem~\ref{thm:manypolytopes} consists in showing that \emph{all} of the many polytopes constructed by Padrol~\cite{Pad13} admit realizations with the many regular triangulations stated in Theorem~\ref{thm:manytriangulations}. 

This makes our proof  of Theorem~\ref{thm:manypolytopes}  more geometric, as opposed to combinatorial, than previous constructions of ``many'' polytopes.
In fact, the combinatorial types of polytopes obtained with our method may depend on choices made along the construction, for example via the choice of realizations used for the Padrol polytopes, which affects what triangulations of them are regular, or via the particular lifting vectors used for the regular triangulations.

\section{Definitions and notation}
We will follow~\cite{Zie95} and \cite{DRS10} for the terminology concerning convex polytopes, point configurations, and triangulations, and we refer the reader to these references for background on these topics.

A \defn{point configuration} is an ordered sequence $\pc =(\p _1, \ldots, \p _n)\in \RR^{d\times n}$. 
We formally consider $P$ a sequence rather than a set since the ordering of the points $\p_i$ is sometimes important, 
but we will slightly abuse notation and write things like $\p_i\in P$, or call the points $\p_i$ the elements of $P$.
In this paper we will usually assume that the points in $P$ are distinct
and in \defn{convex position}, where the latter means that all of them are vertices of the polytope $\conv(\pc)$.
We say that $\pc$ is \defn{$k$-neighborly} if any subset of $k$ points is the vertex set of a face of~$\conv(\pc)$, and just \defn{neighborly} if it is $\ffloor{d}{2}$-neighborly; the latter makes sense since the simplex is the only $d$-polytope that is more than $\ffloor{d}{2}$-neighborly.

A \defn{triangulation} $\tr$ of $\pc$ is a simplicial complex on a subset of $[n]$ such that
\begin{enumerate}[(i)]
\item $\bigcup_{F\in \tr} \conv(\set{\p_i}{i\in F}) = \conv(\pc)$,
\item for all $F, F'\in T$, $\conv(\set{\p_i}{i\in F})\cap\conv(\set{\p_i}{i\in F'})$ is a common face of $\conv(\set{\p_i}{i\in F})$ and $\conv(\set{\p_i}{i\in F'})$.
\end{enumerate}
More generally, a \defn{subdivision} of $\pc$ is a collection $\tr$ of subsets of $[n]$, closed under taking faces (if $F\in \tr$ then the set of indices of points in $\pc$ in a face of $\conv(\set{\p_i}{i\in F})$ must be in $\tr$ too), that fulfills the two conditions above. See \cite[Sec.~2.3]{DRS10} for details.

A subdivision $\tr$ of $\pc$ is \defn{regular} if there is a lifting vector $\p[w] \in \RR^{[n]}$ such that 
for any $F\in \tr$, $\conv(\set{(\p_i, \p[w](i))}{i\in F})$ is a lower face of $\conv(\set{(\p_i, \p[w](i))}{i\in [n]})$. Here, we call a face $F$ of a polytope~$P$ in $\RR^{d+1}$ \defn{lower} if its outer normal cone contains a vector with last coordinate negative. That is, if there is a functional $\p[c]\in \RR^{d+1}$ with $c_{d+1}<0$ that is maximized on $F$. 
We denote \defn{$\regtriang{\pc}$} the set of regular triangulations of $\pc$. 
For a triangulation $\tr$, we call \defn{cells} its maximal faces.
(These are sometimes called facets, but we reserve the word \defn{facet} for facets of a polytope). 

A point $\p[q]\notin \pc$ is said to be in \defn{general position} with respect to $\pc$ if no hyperplane spanned by points of $P$ contains $\p[q]$, and in \defn{very general position} with respect to $\pc$ if moreover no small perturbation of $\p[q]$ changes $\regtriang{\pc\cup\{\p[q]\}}$. 
An argument similar to that in \cite[Part 2]{Ath99} shows that configurations in very general position form a dense open subset of the space of all point configurations.

Two points $\p_i, \p_j\in \pc$ are said to be \defn{triangulation-inseparable} in $\pc$ if we have that 
\begin{itemize}
\item[(i)] $\regtriang{\pc\setminus\{\p_i\}} = \regtriang{\pc\setminus\{\p_j\}}$ up to relabeling $j$ to $i$, and 
\item[(ii)]  for any $\tr\in \regtriang{\pc\setminus\{\p_i\}}$ there is a lifting vector $\p[w]\in \RR^{[n]}$ which restricted to both $\pc\setminus\{\p_i\}$ and $\pc\setminus\{\p_j\}$ produces $\tr$ as a regular triangulation (up to relabeling $j$ to $i$).
\end{itemize}

Let $\p$ be a vertex of $\conv(\pc)$. We define \defn{$\pc/\p$} to be any point configuration obtained as the intersection of the half-lines positively spanned by $\set{\p'-\p}{\p'\in \pc\setminus\{\p\}}$ with an affine hyperplane that does not contain $\p$ and intersects all these half-lines. Following~\cite[Definition 4.2.9]{DRS10} we call $\pc/\p$ the \defn{contraction} of $\pc$ at the point $\p$.
All the configurations that can be obtained as $\pc/\p$ have the same triangulations and the same regular triangulations.
In fact, regular triangulations of $\pc/\p$ are exactly the links at $\p$ of regular triangulations of $\pc$~\cite[Lemmas 4.2.20 and 4.2.22]{DRS10}.
Here, the \defn{link} of a triangulation $\tr$ at a point~$\p_i$, which we denote \defn{$\tr/\p_i$}, is defined as
\[
\tr/\p_i:=\set{F\subset [n]\setminus \{i\}}{F\cup \{i\} \in \tr}.
\]

\section{Many regular triangulations}

The main idea of our construction of configurations with a large number of regular triangulations is to split a point into two triangulation-inseparable points and to estimate the number of regular triangulations generated after this operation. This is inspired by the study of triangulations of cyclic polytopes done in~\cite{Rambau1997,RambauSantos2000}.

First, we show that we can indeed obtain triangulation-inseparable pairs by such a splitting.

\begin{lemma}\label{lem:triang_insep}
Let $\pc$ be a point configuration in $\RR^d$ and $\p\in \pc$ in very general position with respect to $\pc\setminus\{\p\}$. Then there is an $\varepsilon >0$ such that $\p$ and $\p'$ are triangulation-inseparable in $\pc\cup\{\p'\}$ for any $\p'\in \ball{\p}{\varepsilon}$ in very general position with respect to $P$.
Here $\ball{\p}{\varepsilon}$ denotes the ball of radius $\varepsilon$ centered at $\p$.
\end{lemma}

\begin{proof}
Up to relabeling, we can assume that $\pc=(\p_1, \ldots, \p_n)$ and $\p=\p_n$. 
By definition of being in very general position with respect to $\pc\setminus\{\p_n\}$, there exists some $\eta>0$ such that $\regtriang{\pc}=\regtriang{\pc\setminus\{\p_n\}\cup\{\p'\}}$ for all $\p'\in \ball{\p_n}{\eta}$. In particular, any such a $\p'$ fulfills the first condition for being triangulation-inseparable with $\p$.

For each regular triangulation $\tr\in\regtriang{\pc}$ we can choose a specific lifting vector $\p[w]_{\tr}\in \RR^{[n]}$ that induces~$\tr$, and choose it so that the point $(\p_n, \p[w]_{\tr}(n))$ is still in general position with respect to the lifted  configuration $\set{(\p_i, \p[w]_{\tr}(i))}{i\in [n-1]}$. Hence there is some $0<\varepsilon_T< \eta$ such that $\set{(\p_i, \p[w]_{\tr}(i))}{i\in [n]}$ and $\set{(\p_i, \p[w]_{\tr}(i))}{i\in [n-1]}\cup\{(\p', \p[w]_{\tr}(n))\}$ have the same faces, for all $\p'\in \ball{\p_n, \varepsilon_T}$. This means that $\p[w]_T$ induces $T$ as a regular triangulation of $\pc\setminus\{p_n\}\cup\{\p'\}$ for all $\p'\in \ball{\p_n}{\varepsilon_{\tr}}$.

If we take $\varepsilon =\displaystyle{\min_{\tr\in \regtriang{\pc}} \varepsilon_T}$, we obtain that $\p$ and $\p'$ are triangulation-inseparable in $\pc\cup\{\p'\}$ for all~$\p'\in \ball{\p}{\varepsilon}$.
\end{proof}

The following result is our main technical lemma, which provides lower bounds for the number of triangulations under the presence of triangulation-inseparable points. The main ideas are illustrated in Example~\ref{ex:movingtriangulation}.

\begin{lemma}\label{lem:inductionstep}
Let $\pc$ be a point configuration  in $\RR^d$ and let $\p\in \pc$ be a vertex of $\conv(\pc)$ that is in very general position with respect to $\pc\setminus\{\p\}$. 
We denote $C$ the minimum number of cells in a regular triangulation of $\pc/p$.

Let $\p'$ be such that $\p$ and $\p'$ are triangulation-inseparable in $\pc\cup\{\p'\}$, $\p'$ is in very general position with respect to $\pc$, and
 $\p'$ is a vertex of $\conv(\pc\cup\{\p'\})$ . Then we have
\[ 
|\regtriang{\pc\cup\{\p'\}}| \geq |\regtriang{\pc}| \times (C+1). 
\]
\end{lemma}
 
\begin{proof}

We denote $\pc'$ the point configuration $\pc\setminus\{\p\}\cup\{\p'\}$.

Let us call a regular triangulation $\tilde{\tr}$ of $\pc\cup \pc'$ \emph{good} if there is a regular triangulation $T$ of $\pc$ such that $T$ and $\tilde{\tr}$ coincide when restricted to $\pc\setminus\{\p\}$. 
Since a regular triangulation of $\pc$ is determined by its restriction to $\pc\setminus \{\p\}$, this definition implicitly gives a map
\[
\phi: \{\text{good triangulations of $\pc\cup \pc'$}\} \to \regtriang{\pc}.
\]
We claim that for every $T\in \regtriang{P}$ we have 
\[
|\phi^{-1}(\tr)|  \ge \size{\tr/\p}+1 \ge C+1, 
\]
where we denote by $\size{L}$ (and call size of $L$) the number of cells of a pure polyhedral complex $L$.
This formula implies the statement.

Let $\tr$ be a regular triangulation of $\pc$. To avoid confusion we denote by $\tr'$ the triangulation $\tr$ but considered as a triangulation of $\pc'$. 
Let $w\in \RR^{\pc\cup \pc'}$ be a lifting vector producing $\tr$ and $\tr'$ when restricted to $\pc$ and $\pc'$, which exists because $\p$ and $\p'$ are triangulation-inseparable. We will assume moreover a genericity condition on $w$ that will be detailed later at item~(8).

For each $t\in\RR$ we consider the following lifting vector $w_t\in\RR^{P\cup P'}$, which varies continuously with $t$:
\begin{itemize}
\item For $q\in P\setminus\{p\}$, $w_t(q) :=w(q)$ is independent of $t$.
\item If $t\le 0$ then $w_t(p):=w(p)$ and $w_t(p'):=w(p')-t$.
\item If $t\ge 0$ then $w_t(p):=w(p)+t$ and $w_t(p'):=w(p')$.
\end{itemize}

Let $T_t$ be the regular subdivision of $P\cup P'$ produced by $w_t$. We have that:

\begin{enumerate}
\item \emph{The restriction of $T_t$ to $P\setminus \{p\}$ coincides with the restriction of $T$:} Indeed, if $\sigma$ is a face of $T$ contained in $P\setminus \{p\}$ then $w$, and hence any $w_t$, sends all of $P\cup P'\setminus \sigma$ above some supporting hyperplane of the lift of $\sigma$; hence, $\sigma$ is a face in $T_t$. Conversely, suppose $\sigma$ is a cell in $T_t$ for some~$t$ that is contained in $P\setminus \{p\}$. If $t\le 0$ then $w_t$, and hence $w$, sends $P\setminus \sigma$ above the hyperplane. Hence, $\sigma$ is a cell of $T$. If~$t\ge 0$ then $w_t$, and hence $w$, sends $P'\setminus \sigma$ above the hyperplane. Hence, $\sigma$ is a cell of $T'$, which restricted to $P\setminus \{p\}$ coincides with $T$.

\item \emph{If $t\le 0$ then for every cell $\sigma\in T_t$,  $\sigma\setminus \{p'\}$ is a face in $T$:} This is because for $t\le 0$ we have that~$w_t$ restricted to $P$ equals $w$, which produces $T$ as a regular triangulation of $P$.

\item \emph{If $t\ge 0$ then for every cell $\sigma\in T_t$,  $\sigma\setminus \{p\}$ is a face in $T'$:} Same proof.

\item \emph{If $T_t$ is not a triangulation for a certain $t$ then every non-simplicial cell is of the form $\tau\cup\{p,p'\}$ where $\tau$ is a cell of $T/p$.} 
Let $\sigma$ be a non-simplicial cell of $T_t$. By claims (2) and (3), $\sigma$ uses both of $p$ and $p'$ and either $\sigma\setminus\{p'\}$ is in $T$
or $\sigma\setminus\{p\}$ is in $T'$. Hence, $\sigma\setminus\{p,p'\}$ is in $T/p=T'/p'$.

\item \emph{Each such $\tau\cup\{p,p'\}$ appears as a cell for at most one value of $t$:}

If $t\le 0$ then the value of $t$ is fixed by the fact that $w_t(p')=w(p')-t$ equals the height at which the lifted hyperplane containing $\tau\cup\{p\}$ meets the vertical line $\{p'\}\times \RR$ ; same, changing $p$ and $p'$, if $t\ge 0$.
Thus, we  at most have  one value of $t$ in $(-\infty,0]$ and one in $[0,\infty)$.
Moreover, there cannot be two values, one negative and one positive. Indeed, if $\tau\cup\{p,p'\}$ is a cell of $T_t$ for $t<0$, then the point $(p', w(p'))$ is below the lifted hyperplane containing $\tau\cup\{p\}$, so $(p, w(p))$ is above the lifted hyperplane containing $\tau\cup\{p'\}$ and there is no $t'>0$ such that $\tau\cup\{p,p'\}$ is a cell of $T_{t'}$. 

\item \emph{Assuming $T_t$ is a triangulation, let $L_t:= T_t/p \setminus p'$ and $L'_t:= T_t/p' \setminus p$. $L_t$ and $L'_t$ are contained in $T/p$ and they are complementary in the sense that their union equals $T/p$ and their intersection is lower dimensional.}
$L_t$ and $L'_t$ are contained in $T/p=T'/p'$ by properties  (2) and (3). They are complementary because every cell $\tau\in T/p$ needs to be joined to one and only one of $p$ and $p'$ to give a cell of $T_t$.

\item \emph{In the limit when $t\to -\infty$ we have that $L_t=T/p$ (and hence $L'_t$ is lower-dimensional) and in the limit $t\to +\infty$ we have that  $L'_t=T/p$ (and hence $L_t$ is lower-dimensional).}
In these limits, $T_t$ equals the triangulation obtained by \emph{placing} point $p'$ (respectively $p$) in $T$ (respectively in $T'$). This implies  $L_t=T/p$ (respectively $L'_t=T/p$). 

\item \emph{We can take $w$ sufficiently generic so that no $T_t$ contains two different non-simplicial cells.} 
Suppose that $(\tau_1, \tau_2)$ is a pair of cells in $T/p$ such that $\tau_1\cup\{p,p'\}$ and $\tau_2\cup\{p,p'\}$ are in $T_t$ for the same value of $t$.
Let $H_1$ and $H_2$ be the two hyperplanes in $\RR^{d+1}$ spanned by the lifts of $\tau_1\cup\{p\}$ and $\tau_2\cup \{p\}$ for that $t$. Our hypothesis implies that $H_1$ and $H_2$ intersect the vertical line $\{p'\}\times \RR$ at the same height (namely, at height $w_t(p')$). If this happens for a sufficiently generic choice of $w$ then the intersection of $H_1$ with $\{p'\}\times \RR$ does not change when slightly perturbing the heights of all points in $\tau_1\setminus \tau_2$: This implies that this intersection point lies in the affine span of (the lifted) configuration $(\tau_1\cap \tau_2)\cup \{p\}$. Hence, $p'$ lies in the affine span of (the original) $(\tau_1\cap \tau_2)\cup \{p\}$, and $p'$ is not in general position.

\item \emph{If $T_t$ is not a triangulation, and $\varepsilon>0$ is small enough, then $\size{L_{t+\varepsilon}} =\size{L_{t-\varepsilon}}-1$ and $\size{L'_{t+\varepsilon}} =\size{L'_{t-\varepsilon}}+1$.} There is a single cell of $T_t$ of the form $\tau\cup\{p,p'\}$. 
If $s$ is in the neighborhood of $t$, all the cells of $T_s$ not contained in $\tau\cup\{p,p'\}$ remain unchanged because they are defined by an open condition on~$s$. For $s<t$, we have that $\tau\cup\{p\}$ is a cell of $T_s$ but $\tau\cup\{p'\}$ is not, because $p'$ is above the hyperplane spanned by $\tau\cup\{p\}$. Similarly, for $s>t$, we have that $\tau\cup\{p\}$ is not a cell of $T_s$ but $\tau\cup\{p'\}$ is. 
\end{enumerate}

Claim (1) says that whenever $T_t$ is a triangulation it is a 
good triangulation and it lies in the preimage of $T$. As we move $t$ continuously from $-\infty$ to $+\infty$ there are finitely many values of $t$ where $T_t$ is not a triangulation, by claims (4) and (5). Of course, outside those values the triangulation $T_t$ is constant, and claim (8)  says that (if $p'$ is in general position and $w$ is generic) at those values the change in the triangulation is a geometric bistellar flip in a cell of the form $\tau\cup\{p,p'\}$. This flip changes the numbers of cells of $T/p$ contained in $L_t$ and in $L'_t$ by one unit, increasing $L'_t$ and decreasing $L_t$ as $t$ increases, by (9).
By property (7) the size of $L'_t$ grows from zero to $\size{\tr/\p}$ as $t$ goes from $-\infty$ to $+\infty$, so we encounter 
at least $\size{\tr/\p}+1$ different good triangulations in the preimage of $\tr$ along the process.
\end{proof}

\begin{figure}[htpb]
        \centering
\def\hepta{\draw(p) -- (p') -- (p1) -- (p2) -- (p3) -- (p5) -- (p4) -- cycle;}
\def\hexa{\draw(p) -- (p1) -- (p2) -- (p3) -- (p5) -- (p4) -- cycle;}

\newcommand{\triang}[3]{%
    \fill[color=RedOrange!30] (p) -- (p#3) -- (p') -- cycle;
    \draw (p) -- (p');
    \draw \foreach \x in {#1} {(p)--(p\x)};
    \draw \foreach \y in {#2} {(p')--(p\y)};
    \draw [thick, RedOrange] (p) -- (p#3);
    \draw [thick, RedOrange] (p#3) -- (p');
    \draw (p3) -- (p5) -- (p4) ;
    \draw [thick, RedOrange] (p#3) node{$\bullet$};
    \draw [BlueViolet] \foreach \x in {#1} {(p\x) node{$\bullet$}};
    \draw [PineGreen] \foreach \x in {#2} {(p\x) node{$\bullet$}};
    \draw (p) node[right]{$p$};
    \draw (p') node[right]{$p'$};
    \draw (p1) node[above]{$1$};
    \draw (p2) node[left]{$2$};
    \draw (p3) node[left]{$3$};
    \draw (p5) node[below]{$4$};
    \draw (p4) node [below]{$5$};

}

\newcommand{\Lt}[3]{%
	\draw (q1)+(0,0.2) node[above]{$L_{t_{#3}}$};
	\draw [BlueViolet] \foreach \x in {#1} {(q\x) node{$\bullet$}};
	\draw \foreach \x in {#1} {(q\x) node[left]{$\x$}};
	\draw [thick, BlueViolet] \foreach \x in {#1} \foreach \y in {#1} {(q\x) -- (q\y)};
	
	\draw (q1')+(0,0.2) node[above]{$L_{t_{#3}}'$};
	\draw [PineGreen] \foreach \x in {#2} {(q\x ') node{$\bullet$}};
	\draw \foreach \x in {#2} {(q\x ') node[left]{$\x$}};
	\draw [thick, PineGreen] \foreach \x in {#2} \foreach \y in {#2} {(q\x ') -- (q\y ')};
}

\begin{tikzpicture}
    \coordinate (p) at (1.1,1.11);
    \coordinate (p') at (1, 1.5);
    \coordinate (p1) at (0.3,1.88);
    \coordinate (p2) at (-0.76,1.54);
    \coordinate (p3) at (-0.76,0.69);
    \coordinate (p5) at (-0.10,0.16);
    \coordinate (p4) at (0.73,0.35);

    \coordinate (q1) at (0, 1.5);
    \coordinate (q2) at (0, 1);
    \coordinate (q3) at (0, 0.5);
    \coordinate (q5) at (0, 0);
    \coordinate (q1') at (1, 1.5);
    \coordinate (q2') at (1, 1);
    \coordinate (q3') at (1, 0.5);
    \coordinate (q5') at (1, 0);
  
\matrix[column sep=0.5cm,row sep=0.5cm]
{   
	\draw (q1)+(0,0.7) node[above]{$T$};

	\fill[BlueViolet!30] (p) -- (p1) -- (p2) -- (p3) -- (p4) -- cycle;
\draw(p) -- (p1) -- (p2) -- (p3) -- (p5) -- (p4) -- cycle;
	\draw[thick, BlueViolet] (p1) -- (p2) -- (p3) -- (p4);
    \draw \foreach \x/\y in {p/p2, p/p3} {(\x)--(\y)};
	   \draw (p) node[right]{$p$};
    \draw (p1) node[above]{$1$};
    \draw (p1) node[BlueViolet]{$\bullet$};
    \draw (p2) node[left]{$2$};
    \draw (p2) node[BlueViolet]{$\bullet$};
    \draw (p3) node[left]{$3$};
    \draw (p3) node[BlueViolet]{$\bullet$};
    \draw (p5) node[below]{$4$};
    \draw (p4) node [below]{$5$};
    \draw (p4) node[BlueViolet]{$\bullet$};
	&&
	
	\draw (q1)+(0,0.7) node[above]{$T_{t_1}$};
    \fill[color=BlueViolet!30] (p) -- (p1) -- (p2) -- (p3) -- (p4) -- cycle;
    \draw [thick, BlueViolet] (p1) -- (p2) -- (p3) -- (p4);
    \triang{2, 3, 4}{}{1};&

    \draw (q1)+(0,0.7) node[above]{$T_{t_2}$};
    \fill[color=BlueViolet!30] (p) -- (p2) -- (p3) -- (p4) -- cycle;
    \draw [thick, BlueViolet] (p2) -- (p3) -- (p4);
    \fill[color=PineGreen!30] (p') -- (p1) -- (p2) -- cycle;
    \draw [thick, PineGreen] (p1) -- (p2);
    \triang{3, 4}{1}{2}; &

    \draw (q1)+(0,0.7) node[above]{$T_{t_3}$};
    \fill[color=BlueViolet!30] (p) -- (p3) -- (p4) -- cycle;
    \draw [thick, BlueViolet] (p3) -- (p4);
    \fill[color=PineGreen!30] (p') -- (p1) -- (p2) -- (p3) -- cycle;
    \draw [thick, PineGreen] (p1) -- (p2) -- (p3);
    \triang{4}{1, 2}{3}; &
    
    \draw (q1)+(0,0.7) node[above]{$T_{t_4}$};
    \fill[color=PineGreen!30] (p') -- (p1) -- (p2) -- (p3) -- (p4)-- cycle;
    \draw [thick, PineGreen] (p1) -- (p2) -- (p3) -- (p4);
    \triang{}{1, 2, 3}{4};\\   
    \draw (q1)+(0,0.2) node[above]{$T/p$};
	\draw [BlueViolet] \foreach \x in {1,2,3,5} {(q\x) node{$\bullet$}};
	\draw \foreach \x in {1,2,3,5} {(q\x) node[left]{$\x$}};
	\draw [thick,BlueViolet] \foreach \x in {1,2,3,5} \foreach \y in {1,2,3,5} {(q\x) -- (q\y)};&&
    \Lt{1, 2, 3, 5}{1}{1}; &
    \Lt{2, 3, 5}{1, 2}{2}; &
    \Lt{3, 5}{1, 2, 3}{3}; &
    \Lt{5}{1, 2, 3, 5}{4};\\
};

\end{tikzpicture}
\caption{A two-dimensional illustration of the proof of Lemma~\ref{lem:inductionstep}. In the left picture, a triangulation of a hexagon. In the rest, the vertex $p$ is split into $\{p,p'\}$, and increasing values of the parameter~$t$ induce different good triangulations of the heptagon $P\cup P'$.
}\label{fig:proof_inductionstep}
\end{figure}
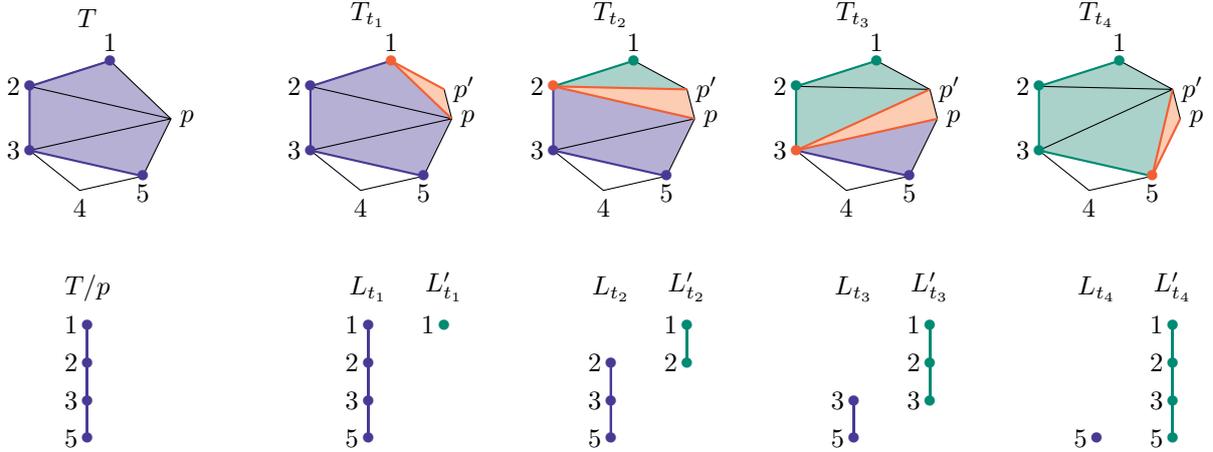

\begin{example}\label{ex:movingtriangulation}
 To illustrate Lemma~\ref{lem:inductionstep}, we use a two-dimensional example (reminiscent of some classical proofs of the recurrence relation for Catalan numbers). It has the advantage of clarity, as it can be easily depicted, see Figure~\ref{fig:proof_inductionstep}. 
 
 When a point $p$ is split into $\{p,p'\}$, the facets of the polytope that were incident to~$p$ are divided into two families, those that remain facets after the splitting, and those that replace $p$ by $p'$. Moreover, new facets containing both $p$ and $p'$ are created. Similarly, the cells incident to $p$ in a triangulation are divided into two families, those containing $p$ and those containing  $p'$, and new cells containing  both $p$ and $p'$ are created. This can be read in the link $T/p$, which is divided into two parts, $L_t$ and $L'_t$, without full dimensional intersection: We have that $F\cup \{p\}\in T_t$ whenever $F\in L_t$, $F\cup \{p'\}\in T_t$ whenever $F\in L'_t$, and that $F\cup \{p,p'\}\in T_t$ whenever $F\in L_t\cap L'_t$.
 
 When $t=-\infty$, we have that $L_t\cap L'_t=L'_t$, and it coincides with the boundary faces of $T/p$ that are incident to~$p'$ in $\conv(P\cup\{p'\})$. As the values of $t$ increase,  $L_t\cap L'_t$ flips successively through each of the simplices of $T/p$, giving rise to different triangulations. At the end, when $t=\infty$, we have that $L_t\cap L'_t=L_t$, and it coincides with the boundary faces of $T/p$ that are incident to~$p$ in $\conv(P\cup\{p'\})$. The number of different triangulations thus created is therefore one more than the number of cells in the link $T/p$.
 
There are some important differences that only appear in higher dimensions. First of all, all triangulations of a polygon are regular, while starting in dimension 3 there are polytopes with non-regular triangulations. Moreover, in dimension two there is only one way to go from $T_{-\infty}$ to $T_{\infty}$ since the link $T/p$ is one-dimensional, while in higher dimensions there are usually several different paths between these two triangulations. The $\size{\tr/\p}+1$ triangulations $T_t$ that we see as $t$ ranges from $-\infty$ to $\infty$ will depend on the relative position of $p$ and $p'$ and the choice of the lifting vector~$w$. Finally, in a polygon, the vertex figure is just a segment (with interior points) and the number $C$ in the statement is always $1$, so the lemma does not give an interesting bound in that case. 
\end{example}

Without any further constraint this lemma is not very useful, as $\conv(\pc/\p)$ could be a simplex and $C=1$. However, a lower bound on $C$  can be proved if we have knowledge on the neighborliness of $\pc/\p$, thanks to the following lemma. 

Recall that for a pure $d$-dimensional simplicial complex $\cC$ and $0\leq j \leq d+1$ we denote \[h_j(\cC)=\sum_{k=0}^j (-1)^{j-k} \binom{d+1-k}{d+1-j} f_{k-1}(\cC),\] where $f_k(\cC)$ is the number of faces of $\cC$ of dimension $k$. The numbers $h_0(\cC), \dots, h_{d+1}(\cC)$, collectively called the \defn{$h$-vector} of $\cC$, are known to be nonnegative in certain special cases, which include $\cC$ being a topological sphere; see~\cite[Chapter 8]{Zie95}.

\begin{lemma}\label{lem:cells_triang_neighb}
Let $d>2$ and $1\leq k \leq d+1$. 
Let $Q$ be a $d$-dimensional simplicial polytope on $n$ vertices. 
Then the number of cells in any triangulation of $Q$ is bounded from below by $h_k(\partial Q)$.

In particular, if $Q$ is $k$-neighborly for $1\leq k \leq \floor{\frac{d}{2}}$, then this number is bounded by:
\[h_k(\partial Q) =\binom{n-d-1+k}{k}.\]
\end{lemma}

 \begin{proof}
 Let $T$ be a triangulation of $Q$. 
 We want a bound on $f_{d}(T)$.
 
 We use the following result from McMullen and Walkup \cite[Thm. 2]{MW71}, cited in a modern version in \cite[Thm. 2.6.11]{DRS10}.
 For any $0\leq j \leq d$, 
 \[h_j(\partial Q)-h_{j-1}(\partial Q) = h_j(T)-h_{d+1-j}(T),
 \]
 where $\partial Q$ is the boundary simplicial complex of $Q$, of dimension $d-1$, and we take $h_{-1}(\partial Q)=0$.
 
 Then we have:
 \begin{align*}
     f_{d}(T) &= \sum_{l=0}^{d+1}h_l(T)\\
     &= h_k(\partial Q)-h_{-1}(\partial Q) +\sum_{j=0}^k h_{d+1-j}(T)+\sum_{l=k+1}^{d+1} h_l(T)\\
     &\geq h_k(\partial Q).
 \end{align*}
 
 The first $h$-coefficients of neighborly polytopes are well known, as they achieve the maximum allowed by the Upper Bound Theorem (\cite[Lemma 2]{McM70}, see also \cite[Lemma 8.26]{Zie95}).
 In particular we have:
 \[
 h_k(\partial Q) = \binom{n-d-1+k}{k}.
 \]
  \end{proof}

\begin{remark}
From the proof one derives that for $1\leq k \leq \floor{\frac{d}{2}}$, a triangulation $T$ has exactly $h_{k}(\partial(Q))$ cells if, and only if, $h_j(T)=0$ for every $j\ge  k+1$. 
This, in turn, is equivalent to all interior cells of $T$  having dimension at least $d-k$. 
\end{remark}

As a consequence of the previous two lemmas we have:

\begin{theorem}\label{thm:induction_regtriang}
Let $P=(p_1, \ldots, p_{n-1}, q)$ be a configuration of $n$ points in very general convex position in $\RR^d$ such that: 
\begin{enumerate}[(i)]
\item for every $d+1\leq i \leq n-1$, $p_{i}$ and $q$ are triangulation-inseparable in $P_i:=(p_1, \ldots, p_i, q)$, and
\item the point configuration $P/q$ is $k$-neighborly.
\end{enumerate}
Then \[|\regtriang{P}| \geq \prod_{m=d}^{n-1} \binom{m-d+k}{k},\]
which is of order $(n!)^{k\pm o(1)}$ for fixed~$k$ and~$d$.

\end{theorem}

\begin{proof}
For $k=0$ the statement is void, therefore we assume that $k\geq 1$. We proceed by induction on $n$. In the base case $n=d+1$ we have that  
$\pc=\pc_d$ is a simplex, with only one regular triangulation, so the result is trivial.

Assume that the theorem is true for $n=m-1$. Note that $P_{m-1}$ satisfies the hypotheses of the theorem. Indeed, the first condition is automatic and the second follows because~$\pc_{m} /q$ is a subset of~$\pc/q$, and a subset of a $k$-neighborly point configuration is still $k$-neighborly.

Now, since $\p_{m}$ and $\p[q]$ are triangulation-inseparable in $\pc_{m}$ by the first hypothesis, we can apply Lemma~\ref{lem:inductionstep} to deduce that 
\[|\regtriang{\pc_{m}}| \geq |\regtriang{\pc_{m-1}}|\times (C_m+1),\]
where $C_m$ is the minimum number of cells in a regular triangulation of $\pc_{m} /q=(\pc /q)\setminus\{\p_{m+1}, \ldots, \p_{n-1}\}$. 
And since $\pc_{m} /q$ is a $k$-neighborly $(d-1)$-dimensional simplicial polytope on $m$ vertices (all points are vertices since it is at least $1$-neighborly), Lemma~\ref{lem:cells_triang_neighb} implies that $C_m\geq \binom{m-d+k}{k}$. 

At the end, using the induction hypothesis we conclude that:
\begin{align*}
|\regtriang{\pc}|
&\geq \prod_{m=d}^{n-1} \binom{m-d+k}{k}\\
&\geq \prod_{m=d}^{n-1} \left( \frac{m-d+k}{k}\right)^{k}\\
&= \left((n-d-1+k)!\right)^{k}\times \frac{1}{\left((k-1)! k^{(n-d)}\right)^{k}} \\
&= \exp(kn\log n + o(n \log n)).
\qedhere
\end{align*}
\end{proof}

The combination of these results provides Theorem~\ref{thm:manytriangulations}: a lower bound of order $(n!)^{\ffloor{d-1}{2} \pm o(1)}$ for the number of regular triangulations of cyclic polytopes in certain realizations. 
Recall that the \defn{cyclic $d$-polytope} with $n$ vertices is a neighborly simplicial polytope that can be realized as the convex hull of $n$ arbitrary points $p_1,\dots,p_n$ along the moment curve $\set{(t,t^2,\dots,t^d)\in \RR^d}{t\in \RR}$. See for example \cite{Zie95} for details.

\begin{proof}[Proof of Theorem \ref{thm:manytriangulations}]
We first fix the last vertex $q=p_n$ on the moment curve and then define the points $p_1, \ldots, p_{n-1}$ consecutively. 
At step $i$, we slide the point $p_i$ along the moment curve until it is close enough to $p_n$ so that Lemma~\ref{lem:triang_insep} implies them to be triangulation-inseparable, after a perturbation of $p_i$ into very general position if needed. For $d\geq 3$, the contraction of the last vertex in a cyclic polytope with~$n$ vertices is a $(d-1)$-dimensional cyclic polytope with~$n-1$ vertices, and in particular $\ffloor{d-1}{2}$-neighborly. Hence Theorem~\ref{thm:induction_regtriang} gives the result.
\end{proof}

It is not clear to us whether cyclic polytopes (or neighborly polytopes in general) do indeed have more triangulations than ``typical'' simplicial polytopes of the same dimension and number of vertices. In fact, in dimension two quite the opposite is true: the convex $n$-gon minimizes the number of triangulations and of regular triangulations among point configurations of $n$ points in general position~\cite{KupavskiiVolostnovYarovikov21,GonzalezSantos21}.

\section{Many polytopes}\label{sec:manypolytopes}

Let us call \defn{polytopal (simplicial) $d$-ball} any (labeled) simplicial complex that can be realized as a regular triangulation of a configuration of points in dimension $d$. By adding a point ``at infinity'' to a polytopal $d$-ball one obtains a polytopal $d$-sphere with one more vertex, and viceversa. Thus,
the number of combinatorially different labeled polytopal $d$-balls with $n$ vertices coincides with the number of combinatorially different labeled simplicial $(d+1)$-polytopes with $n+1$ vertices.

On the other hand, if two simplicial polytopes are combinatorially different  then no triangulation of the first can be combinatorially equal to one of the second, because we can recover the boundary complex of a simplicial polytope from any of its triangulations. Hence:

\begin{lemma}
\label{lem:triangs_to_polytopes}
If $\pc_1, \dots, \pc_N$ are configurations of dimension $d$ and size $n$  in convex and general position and with combinatorially different convex hulls, then there are at least
\[
\sum_{i=1}^N |\regtriang{\pc_i}|
\]
combinatorially different labeled simplicial $(d+1)$-polytopes with $n+1$ vertices.
\end{lemma}

In this section we show that not only cyclic polytopes but all the \emph{Gale sewn} polytopes introduced in~\cite{Pad13}  fulfill (in certain realizations) the conditions of Theorem~\ref{thm:induction_regtriang}. This provides us with a large family of polytopes with many regular triangulations, to which we can apply Lemma~\ref{lem:triangs_to_polytopes} and obtain even more polytopes. 

In order to have a self-contained presentation, we give in  Section~\ref{sub:lextechnical} all the definitions and lemmas that are used in the proofs of the constructions in the Section~\ref{sub:manypolytopes}. Most of the contents of the latter can be traced back to \cite{Pad13,GP16}, but observe that the presentation in~\cite{Pad13} is formulated in the Gale dual setting of extensions while ours, and the one in \cite{GP16}, is already formulated in a primal setting of liftings.

\subsection{Lexicographic liftings}
\label{sub:lextechnical}

A central tool for our  construction are lexicographic liftings, which are a way to derive $(d+1)$-dimensional point configurations from $d$-dimensional point configurations.

\begin{definition}\label{def:lexicographiclifting}
A \defn{positive lexicographic lifting} of a point configuration $P=(p_1,\dots,p_n)\subset\RR^d$ (with respect to the order induced by the labels) is any configuration $\wh P=(\wh p_1,\dots,\wh p_n,\wh q)$ of $n+1$ labeled points in~$\RR^{d+1}$ such that:
\begin{enumerate}[(i)]
\item $\wh q$ is a point in the  halfspace $x_{d+1}>0$,
 \item for $1\leq i\leq n$, the point $\wh p_i$ lies in the half-line  from $\wh q$ through $(p_i,0)$,
 \item \label{it:hyperplanes} for $d+2\leq i\leq n$, and for every hyperplane $H$ spanned by $d+1$ points taken among $\left\{\wh p_{{1}},\dots,\wh p_{{i-1}}\right\}$, 
 the points $\wh q$ and $\wh p_i$ lie on the same side of $H$.
\end{enumerate}
\end{definition}

\begin{remark}\label{rmk:epsilon_lex_lift}
Positive lexicographic liftings exist for every point configuration, and are a special case of the lexicographic liftings produced with a sign vector in $\{+, -\}^n$, as defined e.g.\ in \cite[Def. 4.1]{GP16}.
One way to construct a positive lexicographic lifting is to choose $\wh q$ arbitrarily with $x_{d+1} >0$ and then take $\wh p_i := (1-\varepsilon_i) \wh q + \varepsilon_i (p_i,0)$ for constants $0 < \varepsilon_n \ll \varepsilon_{n-1} \ll\dots \ll  \varepsilon_1$. See Figure~\ref{fig:lifting}.
\end{remark}

\begin{figure}[ht]
\centering
\begin{subfigure}[b]{0.3\textwidth}
\includegraphics[width=.9\linewidth]{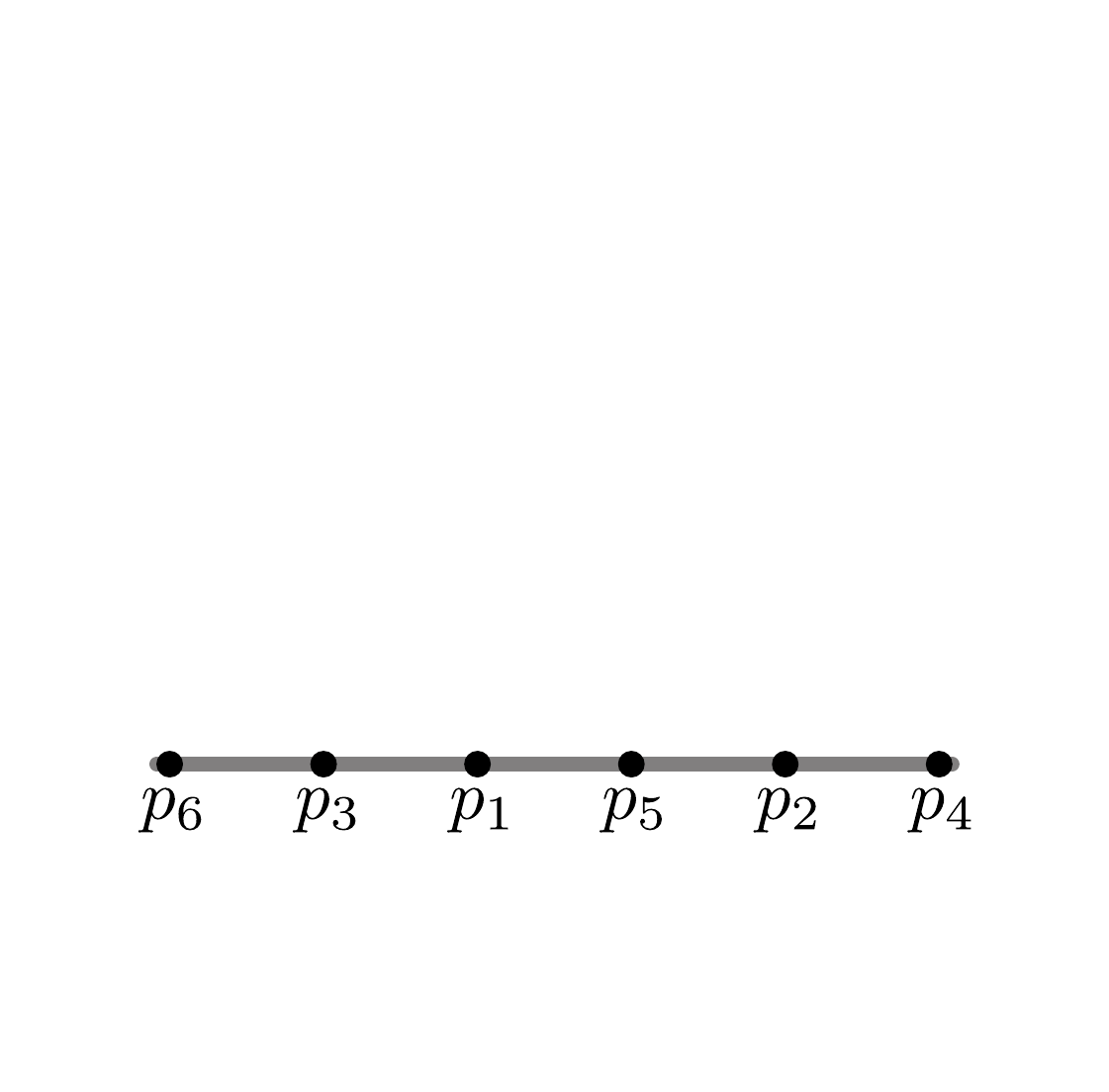}
\caption{$P$}
\label{fig:lifting1}
\end{subfigure}\quad
\begin{subfigure}[b]{0.3\textwidth}
\includegraphics[width=.9\linewidth]{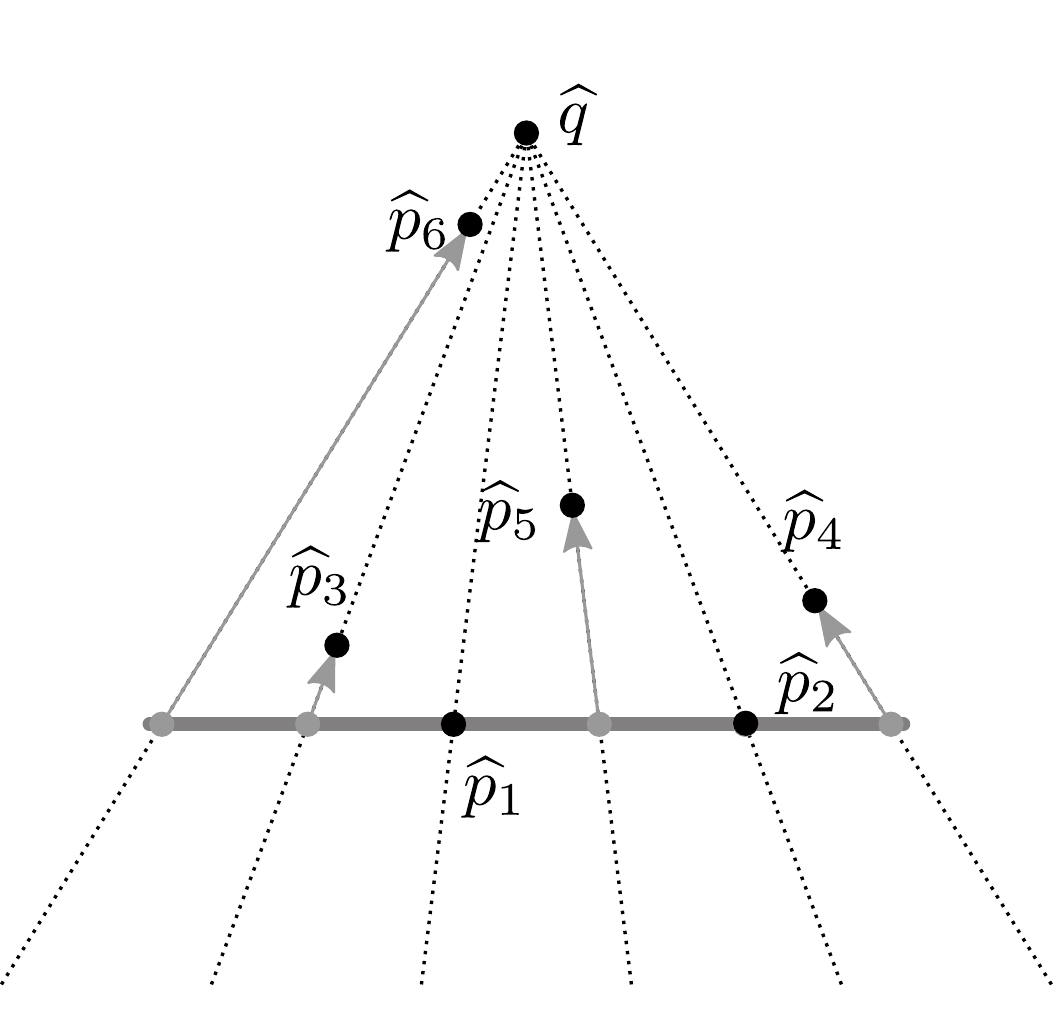}
\caption{$\wh P$}\label{fig:lifting2}
\end{subfigure}\quad
\begin{subfigure}[b]{0.3\textwidth}
\includegraphics[width=.9\linewidth]{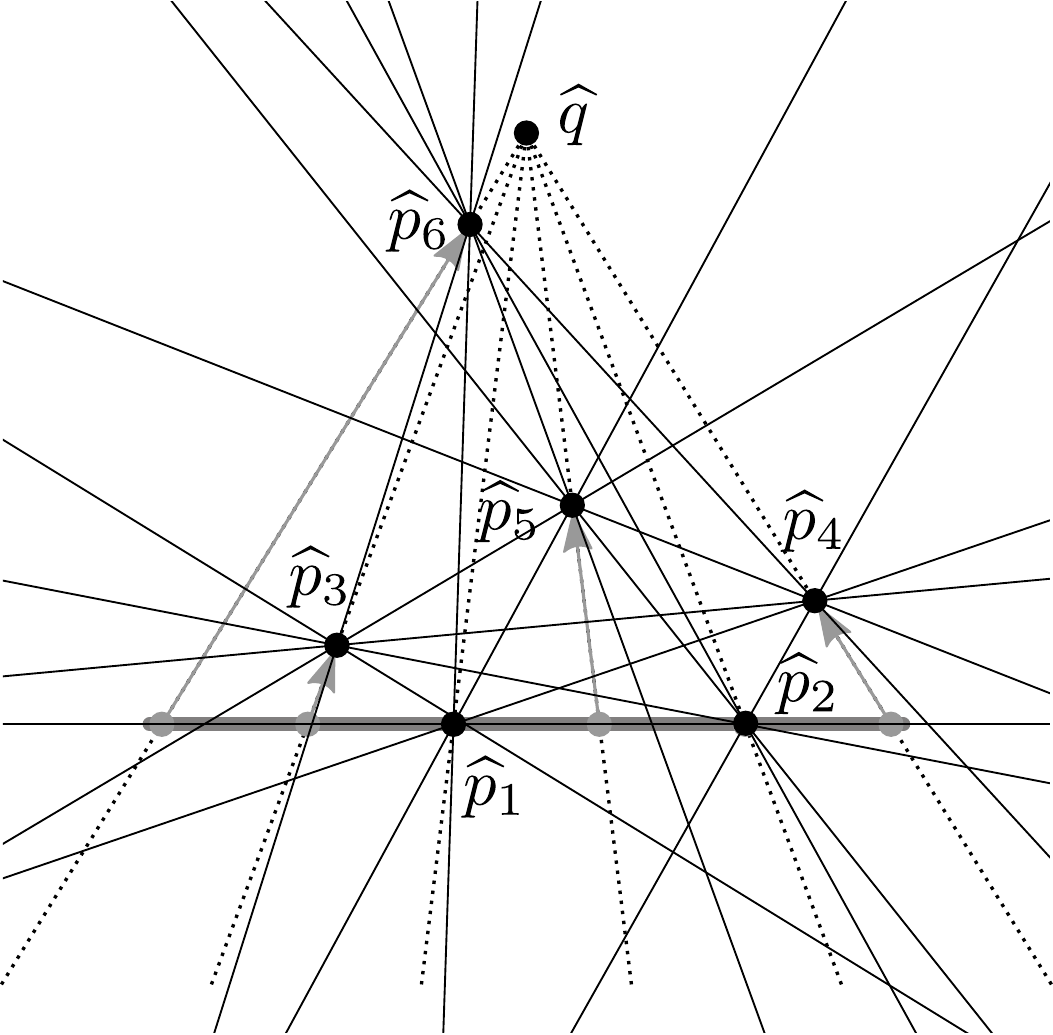}
\caption{Checking \ref{it:hyperplanes}}\label{fig:lifting3}
\end{subfigure}
\caption{A positive lexicographic lifting $\wh P\subset \RR^2$ of a configuration $P\subset \RR^1$.}
\label{fig:lifting}
\end{figure}

The faces of $\conv(\wh P)$ that do not contain $\wh q$ give a particular subdivision of $P$ that is called the \defn{placing}, or \defn{pushing}, triangulation. We refer the reader to Section 4.3.1 of \cite{DRS10} for more details.

\begin{definition}
A face $F$ of a polytope $Q$ is \defn{visible} from a point $p\in \RR^d$ if there is an affine functional that is zero on $F$, strictly positive on $p$ and strictly negative on $Q\setminus F$. 
$F$ is \defn{hidden} from $p$ if there is an affine functional that is zero on $F$ and strictly negative both on $p$ and on $Q\setminus F$. 
Note that a face that is not a facet can be both visible and hidden from $p$, and if $p$ is in general position with respect to $Q$ and $p\notin Q$, then any face of $Q$ (even facets) is either visible or hidden from $p$.

Let $\pc=(\p_1, \ldots, \p_n)$ be a point configuration in general position in $\RR^d$. We denote $P_i:=(\p_1, \ldots, \p_i)$.
The \defn{placing triangulation} $\tr_n$ of $\pc_n$ is defined iteratively by taking for $\tr_1$ the singleton $\{1\}$ and for $\tr_i$ the union of the faces of $\tr_{i-1}$ with all simplices of the form $F\cup\{i\}$ where $F$ gives a face of $\conv(\pc_{i-1})$ that is visible from $\p_i$. 
$\tr_i$ is the only triangulation of $\pc_i$ that contains $\tr_{i-1}$. 
The $\defn{pulling}$ triangulation of $\pc$ is the union of all simplices that give proper faces of $\conv(\pc)$ and all $F\cup\{n\}$ where $F \subseteq [n-1]$ gives a proper face of $\conv(\pc)$. 
(Proper faces are those different from the whole polytope).

\end{definition}

\begin{lemma}\label{lem:faces_lexlifting}
Let $\wh P=(\wh p_1,\dots,\wh p_n,\wh q)$ be a positive lexicographic lifting of the point configuration $P=(p_1,\dots,p_n)\subset\RR^d$ in convex position. 
For $i\in [n]$ we denote $P_i := (p_1, \ldots, p_i)$ and $\wh P_i := (\wh p_1, \ldots, \wh p_i)$. 
Then:
\begin{enumerate}[(i)]
\item\label{it:lexlifting_faces_hidden} The faces of $\conv(\wh P_n)$ that are hidden from $\wh q$ are exactly the liftings of faces of the placing triangulation of $P_n$.
\item \label{it:lexlifting_faces_visible} The faces of $\conv(\wh P_n)$ that are visible from $\wh q$ are exactly the liftings of faces of the pulling triangulation of $P_n$.
\item\label{it:lexlifting_faces_intermediate} For $i\in [n-1]$, the faces of $\conv(\wh P_i)$ that are hidden, resp. visible, from $\wh p_{i+1}$ coincide with the faces that are hidden, resp. visible, from $\wh q$.
\item \label{it:lexlifting_faces_q} The faces of $\conv(\wh P)$ are exactly the faces of $\conv(\wh P_n)$ that are hidden from $\wh q$, which are the liftings of faces of the placing triangulation of $P_n$, and all $\conv(\set{\wh p_i}{i\in F}\cup\{\wh q\})$ where $F$ gives a face of $\conv(P_n)$.
\end{enumerate}
\end{lemma} 

\begin{proof}

Items \ref{it:lexlifting_faces_hidden} and \ref{it:lexlifting_faces_visible} are reformulations of \cite[Lemma 4.3.4]{DRS10} and \cite[Lemma 4.3.6]{DRS10}, which correspond to the case where the point $\wh q$ is ``at infinity''. 
In that case, the faces of $\conv(\wh P_n)$ hidden from $\wh q$ correspond to the  lower faces of $\conv(\wh P_n)$, thus to faces of the corresponding induced regular subdivision of $P_n$. 
The faces of $\conv(P_n)$ visible from $\wh q$ correspond to the lower faces of the lifting of $P_n$ induced by the opposite (negative) heights. 
Then, in both \cite[Lemma 4.3.4]{DRS10} and \cite[Lemma 4.3.6]{DRS10} where we take the opposite heights, the condition on the constant $c_0$ and the heights amounts to asking that the lifting is a positive lexicographic lifting. 

Item \ref{it:lexlifting_faces_intermediate} follows from the definitions and the fact that a face of a polytope $Q$ is hidden, resp. visible, from a point $p$ if and only if it is contained in a facet of $Q$ that is hidden, resp. visible from $p$.

For \ref{it:lexlifting_faces_q}, notice that the faces of $\conv(\wh P)$ that do not contain $\wh q$ are exactly the faces of $\conv(\wh P_n)$ hidden from $\wh q$. The same argument as before shows that they form the placing triangulation of $\wh P_n$. If $F\subseteq [n]$ is such that $F\cup\{\wh q\}$ gives a face of $\wh P$, let $h$ be a supporting hyperplane of this face. Then the intersection of $h$ with the hyperplane $w_{d+1}=0$ is a supporting hyperplane of the face given by $F$ for $\conv(P_n)$ in~$\RR^d\times\{0\}$.
\end{proof}

\begin{corollary}\label{cor:combitypelift}
Let $P=(p_1, \dots, p_n)\subset \RR^d$ be a point configuration in convex position. 
Let $\wh P=(\wh p_1,\dots,\wh p_n,\wh p_{n+1})$
 be a positive lexicographic lifting of $P$ and let $\wwh{P}=(\wwh p_1 ,\dots,\wwh p_n,\wwh p_{n+1}, \wwh p_{n+2})$ 
be a positive lexicographic lifting of $\wh P$, with respect to the same order. Then 
the combinatorial type of $\conv(\wwh P)$ is completely determined by (the oriented matroid of) the point configuration~$P$.
\end{corollary}

\begin{proof}
According to Lemma \ref{lem:faces_lexlifting} \ref{it:lexlifting_faces_q}, the faces of $\conv(\wwh P)$ are the liftings of faces of the placing triangulation of $\wh P_{n+1}$ and all $\conv(\set{\wwh p_i}{i\in F}\cup\{\wwh p_{n+2}\})$ where $F$ gives a face of $\conv(\wh P_{n+1})$. 
The definition of the placing triangulation and Lemma \ref{lem:faces_lexlifting} \ref{it:lexlifting_faces_hidden}, \ref{it:lexlifting_faces_visible}, \ref{it:lexlifting_faces_intermediate} imply that the placing triangulation of $\wh P_{n+1}$ (and thus also the faces of $\conv(\wh P_{n+1})$) is determined by the placing and pulling triangulations of the $P_i$.
\end{proof}

If one starts with a $0$-dimensional point configuration (that is a point repeated multiple times), and then perfoms a sequence of positive lexicographic liftings always with respect to the same order, then one obtains a cyclic polytope. If the order is altered at each step, then many combinatorial types of polytopes are obtained, but not necessarily neighborly.
Moreover, different lifting orders might give rise to equivalent polytopes. However, if one restricts to changing the order of the lifting only every two dimensions, then neighborliness is preserved and the combinatorial type can be controlled.
This is used in~\cite{Pad13} to construct many neighborly polytopes. The original presentation in \cite{Pad13} is in terms of lexicographic extensions of the Gale dual, but we refer to the following primal version for liftings taken from~\cite{GP16}. We repeat the main ideas of that proof for the reader's convenience.

\begin{theorem}[{\cite[Theorem 5.5(i)]{GP16}}]\label{thm:neighborlylift}
 Let $P=(p_1, \dots, p_n)\subset \RR^d$ be a $k$-neighborly point configuration in general position. Let $\wh P=(\wh p_1,\dots,\wh p_n,\wh p_{n+1})$
 be a positive lexicographic lifting of $P$ and let $\wwh{P} =(\wwh p_1,\dots,\wwh p_n,\wwh p_{n+1}, \wwh p_{n+2})$ 
be a positive lexicographic lifting of $\wh P$, with respect to the same order. Then 
$\wwh P$ is $(k+1)$-neighborly.
\end{theorem}

\begin{proof}
Let $S$ be a subset of $[n]$ of size $k$ or $k-1$. Then $\set{p_i}{i\in S}$ is the vertex set of a face of $\conv(P)$. Hence, it follows from Lemma~\ref{lem:faces_lexlifting} \ref{it:lexlifting_faces_q} that $\set{\wh p_i}{i\in S}\cup\{\wh p_{n+1}\}$ is the vertex set of a face of $\conv(\wh P)$.
The same reasoning shows that any subset of $\wwh P$ of size $k+1$ that contains $\wwh p_{n+2}$ or $\wwh p_{n+1}$ is the vertex set of a face of $\wwh P$. 

For the remaining cases, let $S$ be a subset of $[n]$ of size $k+1$. We want to show that $\set{\wwh p_i}{i\in S}$ is the vertex set of a face of $\conv(\wwh P)$. 
Let $m\leq n$ be the largest element of~$S$. Denote $P_m=(p_1, \dots, p_m)$ and $\wh P_m=(\wh p_1,\dots,\wh p_m)$. We have that $S\setminus\{m\}$ gives a face of $\conv(P_{m-1})$ by neighborliness, thus~$S\setminus\{m\}$ gives a face of the pulling triangulation of~$P_{m-1}$, thus $S\setminus\{m\}$ gives a face of~$\wh P_{m-1}$ visible from $\wh p_{m}$ by Lemma~\ref{lem:faces_lexlifting} \ref{it:lexlifting_faces_visible}, thus~$S$ gives a face of the placing triangulation of~$\wh P_m$, and thus~$S$ gives a face of the placing triangulation of~$\wh P$. It follows from Lemma \ref{lem:faces_lexlifting} \ref{it:lexlifting_faces_q} that~$S$ gives a face of~$\wwh P$. 
\end{proof}

The following lemma allows us to prove that the combinatorial type can be controlled without explicitly using the rigidity of neighborly oriented matroids of odd rank as it was originally done in \cite[Proposition~6.7]{Pad13}.

\begin{lemma}\label{lem:typecontrolled}
 Let $P=(p_1, \dots, p_{n})$ be an $r$-neighborly point configuration in even dimension $d=2r$ such that $n>d+2$. Let $\wh P=(\wh p_1,\dots,\wh p_{n},\wh p_{n+1})$ be a lexicographic lifting of $P$ and let $\wwh{P}=(\wwh p_1,\dots,\wwh{p}_{n},\wwh p_{n+1}, \wwh p_{n+2})$ be a positive lexicographic lifting of~$\wh P$, with respect to the same order. 
Then $n$ is the only index $k\in [n]$ such that the double contraction $\wwh{P}/\{\wwh p_{n+1}, \wwh p_k\}$ is $r$-neighborly.
\end{lemma}

\begin{proof}
For $r\geq 1$, we know that $\wwh P$ is $2$-neighborly, so all pairs $\{n+1, k\}$ for $k\in [n]$ give edges of $\wwh P$, and all points of the configuration $\wwh P/\{\wwh p_{n+1}\}$ are vertices. This justifies that the double contraction $\wwh{P}/\{\wwh p_{n+1}, \wwh p_k\}$ is well-defined. 
If $d=r=0$, $\wwh P/\{\wwh p_{n+1}\}$ is a $1$-dimensional configuration of points ordered linearly $n, n-1, \ldots, 2, 1, n+2$. Thus, the double contraction is well-defind only for $k=n+2$ and $k=n$ and we already have the result of the lemma.

Note that $P$ is a realization of $\wwh P/\{\wwh p_{n+2}, \wwh p_{n+1}\}$.
It follows from the definition of contraction that a set $S\subseteq [n]\setminus\{k\}$ gives a face of $\conv(\wwh{P}/\{\wwh p_{n+1}, \wwh p_k\})$ if and only if $S\cup\{n+1, k\}$ gives a face of $\conv(\wwh P)$. 

We denote $\wh P_i:= (\wh p_1, \ldots, \wh p_i)$ for $i\in [n]$.

We first show that $\wwh{P}/\{\wwh{p}_{n+1}, \wwh{p}_{n}\}$ is $r$-neighborly. 
Let $S\subseteq [n+2]\setminus\{n+1, n\}$ be a subset of cardinality~$r$. 
If $S$ contains $n+2$, we define $S':=S\cup\{n\}\setminus\{n+2\}$. $S'$ is a subset of $[n]$ of cardinality $r$, hence it defines a face of $\conv(P)$ and $S\cup\{n+1, n\}=S'\cup\{n+1, n+2\}$ indeed defines a face of $\conv(\wwh P)$. 
If $S$ does not contain $n+2$, then it is a subset of $[n]$ of cardinality $r$ and hence it gives a face of $\conv(P)$.    
Thus, $S\cup\{n\}$ gives a face of the pulling triangulation of $P$, and by Lemma~\ref{lem:faces_lexlifting}\ref{it:lexlifting_faces_visible} a face of $\conv(\wh P_n)$ that is visible from $\wh p_{n+1}$.
Therefore, $S\cup\{n, n+1\}$ gives a face of the placing triangulation of $\conv(\wh P)$, thus a face of $\conv(\wwh P)$.

Now, let $k$  be an element of $[n-1]$. To show that $\wwh P/\{\wwh p_{n+1}, \wwh p_k\}$ is not $r$-neighborly, we will exhibit a subset $S\subseteq [n]\setminus\{n+1, k\}$ of cardinality $r$ such that $S\cup\{n+1, k\}$ does not give a face of $\conv(\wwh P)$. 
Since $n > d+2$, we can find a subset $W$ of $[n]$ of cardinality $d+2=2(r+1)$ that contains $k$ but not $n$. 
Radon's theorem implies that there is a partition of $W$ into two subsets $W_1$ and $W_2$ such that $\conv(\set{p_i}{i\in W_1})\cap\conv(\set{p_j}{j\in W_2})\neq \emptyset$. 
In particular, $W_1$ and $W_2$ do not give faces of $\conv(P)$. 
Since $P$ is $r$-neighborly, $W_1$ and $W_2$ necessarily have at least $r+1$ elements, so they are both exactly of cardinality $r+1$. 
(This is where the assumption of even dimension is used). 
We define $T$ to be the $W_i$ that contains $k$, and $S:=T\setminus\{k\}$.
Since $T$ does not give a face of $\conv(P)$ and does not contain $n$, it does not give a face of $\wh P_n$ that is visible from $\wh p_{n+1}$. 
Hence, $T\cup\{n+1\}=S\cup\{n+1, k\}$ does not give a face of the placing triangulation of $\wh P$. 
However, all faces of $\wwh P$ not containing $n+2$ must be faces of the placing triangulation of $\wh P$ by Lemma~\ref{lem:faces_lexlifting}\ref{it:lexlifting_faces_q}.
Thus, $S\cup\{n+1, k\}$ does not give a face of $\wwh P$.
\end{proof}

\begin{corollary}[{\cite[Proposition 6.1]{Pad13} and \cite[Lemma 6.1]{GP16}}]\label{cor:typecontrolled}
Let $P=(p_1, \dots, p_{n})$ be an $r$-neighborly point configuration in even dimension $d=2r$. 
Then there are at least $\frac{n!}{(d+2)!}$ distinct labeled combinatorial types of $(d+2)$-polytopes 
with $n+2$ vertices obtained by the following construction:

\begin{itemize}
\item Choose a permutation $\sigma$ of $n$.
\item Define the point configuration $P^\sigma=(p_{\sigma(1)}, \ldots, p_{\sigma(n)})$.
\item Let $\wh P^\sigma$ be a positive lexicographic lifting of $P^\sigma$ and let $\wwh P^\sigma=(\wwh p_{\sigma(1)}, \ldots, \wwh p_{\sigma(n)}, \wwh p_{n+1}, \wwh p_{n+2})$ be a positive lexicographic lifting of $\wh P^\sigma$.
\item Define $\wwh P =(\wwh p_1, \ldots, \wwh p_n, \wwh p_{n+1}, \wwh p_{n+2})$.
\item Take the convex hull $\conv(\wwh P)$.
\end{itemize}

\end{corollary}

\begin{remark}
In fact, \cite[Lemma 6.1]{GP16} gives a bound improved by a factor $n+1$, but this does not change the asymptotics of the bound on the total number of polytopes.
\end{remark}

\begin{proof}
Let $\wwh P=(\wwh p_1, \ldots, \wwh p_n, \wwh p_{n+1}, \wwh p_{n+2})$ be a point configuration in $\RR^{d+2}$ obtained as in the statement, with a permutation $\sigma$ that we do not know. 
We will show that we can recover $\sigma(n), \sigma(n-1), \ldots, \sigma(d+3)$ from $P$ and the face lattice of $\conv(\wwh P)$. This implies that distinct choices for $\sigma(n), \sigma(n-1), \ldots, \sigma(d+3)$ give distinct labeled combinatorial types $\conv(\wwh P)$,
and there are $\frac{n!}{(d+2)!}$ such choices. 

We will consecutively recover the values of $\sigma(m)$ starting from $m=n$ until $m=d+3$. Suppose that we have already recovered $\sigma(n), \sigma(n-1), \ldots, \sigma(m+1)$ for some $d+3 \leq m \leq n$. We consider the point configuration $\wwh P_{m} := \wwh P \setminus\{\wwh p_{\sigma(n)}, \ldots, \wwh p_{\sigma(m+1)}\}$ (where we abuse notation for the labels but the only important thing is to record the last two points). 
It follows from Corollary \ref{cor:combitypelift} that its combinatorial type is well defined, because it is obtained as the relabeling of the double lifting of the point configuration $(p_{\sigma(1)}, \ldots, p_{\sigma(m)})$ (with the two additional points $\wwh p_{n+1}$ and $\wwh p_{n+2}$).  
Moreover, since the point configuration $(p_{\sigma(1)}, \ldots, p_{\sigma(m)})$ is $r$-neighborly, it follows from Lemma~\ref{lem:typecontrolled} that we can recover $\sigma(m)$ as the only index $k\in [m]$ such that $\wwh P_{m}/\{\wwh p_{n+1}, \wwh p_k\}$ is $r$-neighborly. 
\end{proof}

\subsection{Construction of many polytopes}
\label{sub:manypolytopes}

We will use the following slight variation of the construction used in~\cite{Pad13} to give a lower bound for the number of polytopes.
\begin{theorem}[{\cite[Theorem 6.8]{Pad13}}]\label{thm:manyneighborly}
 The number of labeled combinatorial types of neighborly $d$-polytopes with $n>d$ vertices obtained from a $0$-dimensional point configuration by a sequence of positive lexicographic liftings (with orders that might change along each step of the sequence) is at least 
 \[(n!)^{\ffloor{d}{2}\pm o(1)}.\]
\end{theorem}

\begin{proof}
We build iteratively sets $\cP_{2k}$ that contain realizations of distinct labeled combinatorial types of neighborly polytopes of dimension $2k$ with $n-d+2k$ vertices. 

We define $\cP_0$ to be the singleton with the degenerate configuration of $n-d$ labeled points in the $0$- dimensional space. 

Suppose that we have constructed $\cP_{2k}$ for some $0\leq k < \ffloor{d}{2}$. 
Let $\cP_{2k+2}$ be the union over all configurations $P\in \cP_{2k}$ of the distinct labeled point configurations obtained from $P$ by relabelings and two positive lexicographic liftings in the same order, as in Corollary~\ref{cor:typecontrolled}. This union is disjoint because if $\wwh P$ is a double lifting of $P$, we can recover the combinatorial type of $P$ by taking $\wwh P/\{\wwh p_{n-d+2k+2}, \wwh p_{n-d+2k+1}\}$. Hence, Corollary~\ref{cor:typecontrolled} gives that $|\cP_{2k+2}|\geq |\cP_{2k}|\times \frac{(n-d+2k)!}{(2k+2)!}$. 
Theorem~\ref{thm:neighborlylift} ensures that the point configurations in $\cP_{2k+2}$ are neighborly.

For $k=\ffloor{d}{2}$ we obtain that:
\begin{align*}
|\cP_{2\ffloor{d}{2}}| & \geq 
\prod_{k=0}^{\ffloor{d}{2}-1} \frac{(n-d+2k)!}{(2k+2)!}\\
&\geq \frac{\left( (n-d)! \right)^{\ffloor{d}{2}}}{\prod_{k=1}^{\ffloor{d}{2}} (2k)!}\\
&=(n!)^{\ffloor{d}{2} + o(1)}.
\end{align*}

If $d$ is odd, instead of taking a pyramid as in \cite[Corollary~6.10]{Pad13}, we do one last positive lexicographic lifting on all the elements of $\cP_{2\ffloor{d}{2}}$ to obtain $(n!)^{\ffloor{d}{2} + o(1)}$ realizations of distinct labeled combinatorial types of $d$-polytopes with $n$ vertices. 
This variant still conserves the number of distinct combinatorial types since we recover the polytopes in $\cP_{2\ffloor{d}{2}}$ by taking the contractions of the last labeled point.
\end{proof}

The combination of these constructions allows us to prove Theorem~\ref{thm:manypolytopes}: The number of different labeled combinatorial types of $d$-polytopes with $n$ vertices for fixed $d>3$ and $n$ growing to infinity is at least $(n!)^{d-2 \pm o(1)}$.

\begin{proof}[Proof of Theorem~\ref{thm:manypolytopes}]
We start by applying Theorem~\ref{thm:manyneighborly}  in dimension $d-1$.
The last step of the construction of the many $(d-1)$-polytopes in that theorem is a positive lexicographic lifting $\wh \pc = (\wh p_1, \ldots, \wh p_{n-1}, \wh q)$ from a $\floor{\frac{d-2}{2}}$-neighborly $(d-2)$-polytope $P$. 

Lemma~\ref{lem:triang_insep} ensures that we can do this lifting step by step so that for every $i$ from $d$ to $n-1$, $\wh p_{i}$ and $\wh q$ are triangulation-inseparable in $(\wh p_1, \ldots, \wh p_i, \wh q)$. Indeed, the value of $\varepsilon_i$ in Remark~\ref{rmk:epsilon_lex_lift} can be taken arbitrarily small. While very general position is not guaranteed by the construction, note that these configurations are in general position, and hence we can do a small perturbation into very general position if needed without changing the combinatorial type. 

Moreover, note that by construction $P$ is the contraction $\wh \pc /\wh q$, and that similarly $( p_1, \ldots, p_i) = (\wh p_1, \ldots, \wh p_i, \wh q)/\wh q$. These contractions are thus $\floor{\frac{d-2}{2}}$-neighborly.

Hence Theorem~\ref{thm:induction_regtriang} applies: each of these polytopes has at least $(n!)^{\ffloor{d-2}{2}n\pm o(1)}$ regular triangulations. Then Lemma~\ref{lem:triangs_to_polytopes} gives us a lower bound of $(n!)^{d-2\pm o(1)}$ labeled simplicial types of $d$-polytopes with $n$ vertices.
\end{proof}

\begin{remark}\label{rmk:notmanyneighborly}
It follows from the construction that all these many $d$-polytopes are $\ffloor{d-1}{2}$-neighborly, because they come from regular triangulations of Padrol's neighborly $(d-1)$-polytopes.

Hence, for odd $d$ our polytopes are neighborly, since in this case $\ffloor{d}{2}=\ffloor{d-1}{2}$.

On the other hand, if $d$ is even then the following lemma shows that we do not improve Padrol's bound on the number of neighborly polytopes, because each of the Padrol polytopes that we use has at most one neighborly triangulation.
\end{remark}

\begin{lemma}\label{lem:neighborlytriang}
A polytope in odd dimension $2k+1$ has at most one triangulation that is $(k+1)$-neighborly. 
\end{lemma}

\begin{proof}
This is a direct consequence of the observation after \cite[Lemma 3.1]{Dey93}, see also \cite[Lemma 8.4.1]{DRS10}: a triangulation of a $d$-polytope is completely determined by its $\ffloor{d}{2}$-skeleton. 
For a triangulation of a $(2k+1)$-polytope, being $(k+1)$-neighborly exactly means that its $k$-skeleton is  complete.
\end{proof}

\bibliographystyle{plain}
\bibliography{biblio_manytriang} 

\end{document}